\documentclass{article}

\usepackage{amsmath,amsthm,amsfonts,amssymb,amsthm}
\usepackage{array}
\usepackage{hyperref}
\usepackage{mathrsfs}

\usepackage[margin=1in]{geometry}

\usepackage{float}
\usepackage{tikz}
\usepackage{tikz-network}
\usetikzlibrary{decorations.pathmorphing, patterns}
\usepackage{subfig}
\newtheorem{theorem}{Theorem}[section]
\newtheorem{lemma}[theorem]{Lemma}

\newtheorem{corollary}[theorem]{Corollary}

\newtheorem{algorithm}[theorem]{Algorithm}

\theoremstyle{definition}

\newtheorem{example}[theorem]{Example}

\newcommand{\MaxMinDist}{maximum inner distance}
\newcommand{\innerDist}{inner distance}
\newcommand{\AdjDist}{adjacent distance}

\newcommand{\MmD}{\ensuremath{\mathrm{MiD}}}
\newcommand{\GrantNumber}{DMS-1852378}

\newcommand{\floor}[1]{\ensuremath{\left\lfloor #1 \right\rfloor}}

\newcommand{\term}[1]{\textbf{#1}}
\newcommand{\ts}[1]{\textsuperscript{#1}}
\title{Distance in Latin Squares}
\author{
Omar A. Garcia \\ 
University of Texas \\
\url{oaceval1501@utexas.edu}
\and
Paige Beidelman \\ 
University of Mary Washington \\ 
\url{pbeidelm@mail.umw.edu} 
\and
Jieqi Di \\ 
Boston College \\ 
\url{dij@bc.edu} 
\and
James Hammer \\ 
Cedar Crest College \\ 
\url{jmhammer@cedarcrest.edu} 
\and 
Mitchel O'Connor \\ 
Whitman College \\ 
\url{oconnoml@whitman.edu}
\and 
Caitlin Owens\\
DeSales University\\
\url{caitlin.owens@desales.edu}
\and 
Yewen Sun \\ 
UC Santa Barbara \\ 
\url{yewen@ucsb.edu}}

\date{\today}

\begin{document}

\maketitle
    
\begin{abstract}
A Latin square of order $n$ is an $n\times n$ array which contains $n$ distinct symbols exactly once in each row and column. We define the adjacent distance between two adjacent cells (containing integers) to be their difference modulo $n$, and inner distance of a Latin square to be the minimum of adjacent distances in the Latin square. By first establishing upper bounds and then constructing squares with said inner distance, we found the maximum inner distance of an $n \times n$ Latin square to be $\left\lfloor\frac{n-1}{2}\right\rfloor$. We then studied special kinds of Latin squares such as pandiagonals (also known as Knut-Vik designs), as well as Sudoku Latin squares. This research was conducted at the REU at Moravian College on Research Challenges of Computational and Experimental Mathematics, with support from the National Science Foundation. 
\end{abstract}

\section{Introduction}
In this paper, we will define and investigate the \term{inner distance} of Latin squares. In previous literature, the term \textit{distance} has been used to denote the number of different cells between two $n\times n$ Latin squares, such as in \cite{ramadurai2016distances, xiao2017construction}, or the literal Euclidean distance between cells such as in \cite{luo2012distance}. In this paper, we ask novel questions about Latin squares and provide preliminary results on what we will refer to as the \textit{inner distance} of a Latin square. 

A \term{Latin square} is an $n \times n$ array with entries from a set $S$ of order $n$, such that each row and each column contains every element in $S$. An element in $S$ is called a symbol of the Latin square \cite{book:Combinatorics}. For the duration of this paper, we will use the integers 1 through $n$ as our symbols.

\begin{figure} [H]
    \centering
    \begin{tikzpicture}[scale=.6]
\draw(0,0)grid(5,5); 
\draw[step=5,ultra thick](0,0)grid(5,5);
\foreach\x[count=\i] in{1, 2, 3, 4, 5}{\node at(\i-0.5,4.5){$\x$};};
\foreach\x[count=\i] in{ 2, 3, 4, 5, 1}{\node at(\i-0.5,3.5){$\x$};};
\foreach\x[count=\i] in{3, 4, \textbf{5}, 1, 2}{\node at(\i-0.5,2.5){$\x$};};
\foreach\x[count=\i] in{4, 5, 1, 2, 3}{\node at(\i-0.5,1.5){$\x$};};
\foreach\x[count=\i] in{5, 1, 2, 3, 4}{\node at(\i-0.5,0.5){$\x$};};
\end{tikzpicture}
    \caption{$5 \times 5$ Latin square}
    \label{fig:LS5}
\end{figure}
We say some set (such as a row, column, set of cells, etc.) is \term{Latin} if it only contains each of the integers in $\{1,2, \dots, n\}$ exactly once. Some constructions of Latin squares are particularly useful for this paper. Let $L_n$ be an $n\times n$ Latin square and let $k$ be an integer such that $\gcd(k,n)=1$.  Then $L_n$ is said to be a \textbf{shift-by-$\mathbf{k}$} Latin square if the symbol for any cell $(i,j)$ also lies in the cell $(i+1 \pmod{n},j+k\pmod{n})$ where $i$ represents the row and $j$ represents the column. When $k=1$, the Latin square is known as a \textit{circulant Latin square}, when $k=-1$, the Latin square is called a \textit{back circulant} Latin square.

Figure \ref{fig:LS5} is an example of a back circulant $5\times 5$ Latin square and Figure \ref{examples for shift by k} contains examples of $5\times 5$ shift-by-$k$ Latin squares for $k=1,2,3$.

\begin{figure}[H]
    \centering
    \begin{tikzpicture}[scale=.6]
\draw(0,0)grid(5,5); 
\draw[step=5,ultra thick](0,0)grid(5,5);
\foreach\x[count=\i] in{1, 2, 3, 4, 5}{\node at(\i-0.5,4.5){$\x$};};
\foreach\x[count=\i] in{5, 1, 2, 3, 4}{\node at(\i-0.5,3.5){$\x$};};
\foreach\x[count=\i] in{4, 5, 1, 2, 3}{\node at(\i-0.5,2.5){$\x$};};
\foreach\x[count=\i] in{3, 4, 5, 1, 2}{\node at(\i-0.5,1.5){$\x$};};
\foreach\x[count=\i] in{2, 3, 4, 5, 1}{\node at(\i-0.5,0.5){$\x$};};
\end{tikzpicture}    
     \qquad
    \begin{tikzpicture}[scale=.6]
\draw(0,0)grid(5,5); 
\draw[step=5,ultra thick](0,0)grid(5,5);
\foreach\x[count=\i] in{1, 2, 3, 4, 5}{\node at(\i-0.5,4.5){$\x$};};
\foreach\x[count=\i] in{4, 5, 1, 2, 3 }{\node at(\i-0.5,3.5){$\x$};};
\foreach\x[count=\i] in{2, 3, 4, 5, 1}{\node at(\i-0.5,2.5){$\x$};};
\foreach\x[count=\i] in{5, 1, 2, 3, 4}{\node at(\i-0.5,1.5){$\x$};};
\foreach\x[count=\i] in{3, 4, 5, 1, 2}{\node at(\i-0.5,0.5){$\x$};};
\end{tikzpicture}
    \qquad
\begin{tikzpicture}[scale=.6]
\draw(0,0)grid(5,5); 
\draw[step=5,ultra thick](0,0)grid(5,5);
\foreach\x[count=\i] in{1, 2, 3, 4, 5}{\node at(\i-0.5,4.5){$\x$};};
\foreach\x[count=\i] in{3, 4, 5, 1, 2 }{\node at(\i-0.5,3.5){$\x$};};
\foreach\x[count=\i] in{5, 1, 2, 3, 4 }{\node at(\i-0.5,2.5){$\x$};};
\foreach\x[count=\i] in{2, 3, 4, 5, 1}{\node at(\i-0.5,1.5){$\x$};};
\foreach\x[count=\i] in{4, 5, 1, 2, 3}{\node at(\i-0.5,0.5){$\x$};};
\end{tikzpicture}
   
    \caption{Left to right: a circulant, Shift-By-2, and Shift-By-3 Latin square}
    \label{examples for shift by k}
\end{figure}

If $L_n$ is a Latin square, use the symbol $(i,j)$ to denote the cell on the $i^{\rm th}$ row and $j^{\rm th}$ column. We write $m_{i,j}$ to denote the entry of the cell $(i,j)$. For example, in the Latin squares in Figure \ref{examples for shift by k}, $m_{2, 1}$ takes values 5, 4, and 3 in the squares on the left, middle, and right, respectively.

Two cells in a Latin square are \term{adjacent} if the cells share an edge horizontally or vertically. The cells $(i_1,j_1)$ and $(i_2,j_2)$ are horizontally adjacent if $i_1=i_2$ and $|j_2-j_1|=1$, and vertically adjacent if $j_1=j_2$ and $|i_2-i_1|=1$. For example in Figure \ref{fig:LS5}, the cell with the bold $5$ is adjacent to cell above with $4$, below with $1$, left with $4$, and right with $1$.

We now introduce our terminology. The \term{\AdjDist} between two cells in a Latin square which contain symbols $u,v,$ is defined to be smallest difference from $u$ to $v$ (or vice versa) modulo $n$, given by $\min\{u-v \pmod{n}, v-u \pmod{n} \}$, such that $0\leq \min\{u-v \pmod{n}, v-u \pmod{n} \} \leq n-1$. In Figure \ref{fig:LS5}, the \AdjDist~between any pair of adjacent cells is 1. In this paper, we simply refer to the adjacent distance between cells as the \textit{distance} between those cells. 

We define the \textbf{\innerDist} of a Latin square to be the minimum of all distances between adjacent cells. For example, all Latin squares shown in Figures \ref{fig:LS5} and \ref{examples for shift by k} have \innerDist~1. In this paper, we establish the maximum value of the inner distance for certain classes of Latin squares, and we will provide bounds for the maximum value in other classes. We will use $\MmD(L_n)$ to denote the maximum inner distance of the set of Latin squares of order $n$, $\MmD(P_n)$ to denote the maximum inner distance of the set of pandiagonal Latin squares of order $n$, and $\MmD(L_{a,b})$ to denote the maximum inner distance of the set of $(a,b)$-Sudoku Latin squares of order $ab$. These classes of Latin squares just mentioned will be defined later in the paper.

Let $n = a \times b$ be the order of a Latin square. We define $B_i = \floor{\frac{i-1}{a}}$ and $S_j = \floor{\frac{j-1}{b}}$. We say the cell $(i,j)$ lies on the $B_i\ts{th}$ \textbf{band} and the $S_j$\ts{th} \textbf{stack}. Each band consists of $a$ consecutive rows, and each stack consists of $b$ consecutive columns. These will be useful when working with Sudoku Latin squares, defined in Section \ref{Sudoku section} (also where we reserve the symbols $a,b$ in this paper). We call the intersections of a band and a stack a \textbf{block}, which is an $a\times b$ region in a Latin square. There are $\frac{n}{a}=b$ bands and $\frac{n}{b}=a$ stacks. If $k = B_i$ and $l = S_j$, we denote the block along the intersection of the $k\ts{th}$ band and $l\ts{th}$ row by $M_{k,l}$.

The following theorems from basic number theory are frequently used in this paper, and so their proofs are emphasized. Let $a$, $b$, and $n$ be integers with $n > 0$. We use $\gcd(a,b)$ to denote the greatest common divisor of $a$ and $b$, and we use the notation $a \pmod{n}$ to note the unique integer $r$ such that $a\equiv r \pmod{n}$, with $1\leq r \leq n$. 

\begin{theorem}\label{coprime increment theorem}\cite{numberTheory}
If $n$ is a natural number, $a$ is an integer, and $k$ is an integer relatively prime to $n$, then 
$\{1,2, \dots, n\}  = \{a +km \pmod{n} \mid 0\leq m\leq n-1\} $.
\end{theorem}
\begin{proof}
Assume two elements from $ \{a +km \pmod{n} \mid 0\leq m \leq n-1\}$ are equivalent for some $m_1,m_2$:
\[a+km_1 \equiv a+km_2 \pmod{n}\implies k(m_1-m_2) \equiv 0 \pmod{n}.\]
Then $n$ divides $k(m_1-m_2)$, and since $\gcd(n,k)=1$, $n$ divides $(m_1-m_2)$. However, since $0\leq m_1,m_2\leq n-1$, we have that $0 \leq |m_1-m_2| \leq n-1$. Therefore $m_1-m_2=0$, and $m_1=m_2$.
\end{proof}

\begin{theorem}\label{modular divides stuff}\cite{numberTheory}
Let $a,b,$ and $n$ be integers, with $n>0$. Suppose for some integer $x$ that $ax\equiv b\pmod{n}$. Then $\gcd(n,a)\mid b$. If $b = 0$, then $\left.\frac{n}{\gcd(n,a)} \right| x$.
\end{theorem}
\begin{proof}
Let $ax\equiv b\pmod{n}$. Then by definition $n$ divides $ax-b$, so that $ax-b = nk$ for some integer $k$. Then $ax - nk = b$, so that $\gcd(n,a)$ divides the left hand side, and so $\gcd(n,a)\mid b$. 

Let $n = \gcd(n,a)n_0$, $a = \gcd(n,a)a_0$, so that we want to show $n_0\mid x$. Note that $\gcd(n_0, a_0)=1$, since the $\gcd$ pulls out all common factors of $a$, $n$. If $b=0$, then we have that $ax = nk$. Dividing by $\gcd(n,a)$ we have that $a_0x = n_0k$. So $n_0$ divides $a_0x$, and since $\gcd(n_0, a_0)=1$, then $n_0\mid x$.
\end{proof}

\begin{theorem}\label{gcd stuff}
Let $a$ be a positive integer, $b$ be an odd positive integer, and $n=ab$. Then $\gcd(n, \frac{n-a}{2}) = a$.
\end{theorem}
\begin{proof}
Let $a$ be a positive integer, $b$ be an odd positive integer, and $n=ab$. Since $\gcd(n, n-a) \mid a$, $a \mid n$, and $a \mid (n-a)$, $\gcd(n, n-a)=a$. Since any divisor of $\frac{n-a}{2}$ is also a divisor of $n-a$, $\gcd(n, \frac{n-a}{2}) \mid \gcd(n, n-a)$. Since $b$ is odd, $b-1$ is even, and so $a$ is still a divisor of $\frac{n-a}{2}=\frac{ab-a}{2}=\frac{a(b-1)}{2}$. Thus, $\gcd(n, \frac{n-a}{2})=a$.
\end{proof}

\section{Maximum Inner Distance}\label{MmD Section} 

We now address the question that this paper is concerned with: What is the \MaxMinDist~that one can obtain for a given Latin square of order $n$?

\subsection{An Upper Bound For the Maximum Inner Distance of Latin Squares}
We begin with a trivial observation on $\MmD(L_n)$. We can observe for $n\leq 2$, there are only two Latin squares of order two (using the symbol set $\{1,2\}$), each with \innerDist~$1$. The Latin square of order one has only one cell, and thus has no defined \innerDist.

\begin{lemma}\label{upperbound}
If $n\geq 3$, then $$\MmD(L_n)\leq \floor{\frac{n-1}{2}}.$$
\end{lemma}
\begin{proof}
Let $n$ be odd. The possible differences between cells in a Latin square of order $n$ are
\[\left\{1, 2, \dots, \frac{n-1}{2}, \frac{n+1}{2}, \dots, n-1\right\}.\]
Let $u,v$ be integers in adjacent cells. If $u-v \geq \frac{n+1}{2}$, then $v-u \leq -\frac{n+1}{2} \equiv \frac{n-1}{2} \pmod{n}$. Since $\frac{n-1}{2} = \floor{\frac{n-1}{2}}$, the largest distance between two adjacent cells is $\floor{\frac{n-1}{2}}$. Then the inner distance of odd ordered squares is at most $\frac{n-1}{2}$.

Let $n$ be even. Thus the same analysis shows that the distance is at most $\frac{n}{2}$. However, for any integer $x\in \{1,2,\dots, n\}$, there is only one integer at maximum distance, namely $x + \frac{n}{2} \equiv x-\frac{n}{2} \pmod{n}$. If $n\geq 3$ and $n$ is even, then for $2\leq j \leq n-1$, $m_{i,j}$ has at least two distinct integers: $m_{i,j-1}$ and $m_{i,j+1}$ adjacent to it. This is due to the cells to the left and right are in the same row $i$ and thus cannot contain the same integer. Then, at least one of $m_{i,j-1}$ and $m_{i,j+1}$ has distance less than or equal to $\frac{n}{2}-1 = \frac{n-2}{2} = \floor{\frac{n-1}{2}}$, and therefore the inner distance of an even ordered square is at most $\floor{\frac{n-1}{2}}$.
\end{proof}
Lemma \ref{upperbound} demonstrates that in the modular world, adding beyond a certain threshold, about $\frac{n}{2}$, is equivalent to subtracting a smaller amount. We will now prove, by construction, that this bound can be achieved.  

\subsection{An Algorithm to Produce Latin Squares With Maximum Inner Distance}
\begin{algorithm}\label{generalized algorithm}
Let $1 \leq c,r\leq n-1$, $R = \frac{n}{\gcd(n,r)},$ and  $C = \frac{n}{\gcd(n,c)}$. Let $\alpha, \beta \in \mathbb{Z}$, such that $\gcd(\beta,c)=1$ and $\gcd(\alpha,r)=1$. 

Fill the cell $(i,j)$ in an $n\times n$ array with the integer
\[m_{i,j} = 1 + (i-1)r + (j-1)c + \alpha\left\lfloor{\frac{i-1}{R}}\right\rfloor + \beta\left\lfloor{\frac{j-1}{C}}\right\rfloor \pmod{n}.\]
\end{algorithm}

\begin{theorem}
Algorithm \ref{generalized algorithm} produces a Latin square.
\end{theorem}

\begin{proof}
Recall that an array of integers from $\{1,2,\dots, n\}$ is a Latin square if and only if the rows and columns are Latin. We show the rows are Latin, and the argument for columns is identical.

Consider an $n\times n$ array created by Algorithm \ref{generalized algorithm}. Assume two integers along the same row, $m_{i, j_1}$ and $m_{i, j_2}$, are equal. Setting the expressions for these two equal and cancelling terms we obtain:
\[(j_1 - 1)c +  \beta\floor{\frac{j_1-1}{C}} \equiv (j_2 - 1)c + \beta\floor{\frac{j_2-1}{C}} \pmod{n}.\]

Rewriting  this expression in terms of stacks, we have that 
\[(j_1 - 1)c + \beta S_{j_1} \equiv (j_2 - 1)c + \beta S_{j_2} \pmod{n}\iff (j_1-j_2)c \equiv \beta (S_{j_2} - S_{j_1}) \pmod{n}.\]

Then by Theorem \ref{modular divides stuff}, $\gcd(n,c)$ divides $\beta(S_{j_2}-S_{j_1})$. Since $\gcd(\beta,c)=1$, the $\gcd(n,c)$ divides $S_{j_2}-S_{j_1}$. However, this can only happen if $S_{j_1} = S_{j_2}$, as there are $\gcd(n, c)$ stacks in a Latin square, and so $0\leq |S_{j_1} - S_{j_2}| \leq \gcd(n, c) - 1$. 

Therefore, the cells are in the same stack, and so our congruence becomes $(j_1-j_2)c\equiv 0\pmod{n}$. By Theorem \ref{modular divides stuff}, $C$ divides $j_1-j_2$. But since there are $C$ columns in a stack, $0\leq |j_1 - j_2| \leq C - 1$, and so $j_2-j_1 = 0$. Therefore, $j_2=j_1$. 
\end{proof}

In Algorithm \ref{generalized algorithm}, the hard coded ones (except those inside the floor function) only serve to place a 1 in the top left corner, for convenience. They can be altered to other integers in $\{1,2,\dots, n\}$ or omitted entirely for the algorithm to work. Additionally, $r$ represents the difference between vertically adjacent cells within the same band and $c$ represents the horizontal distance between cells within the same stack. The offsets $\alpha$ and $\beta$ in the algorithm are only necessary for values of $r$ and $c$ which share common factors with $n$, to avoid repeats of entries which are in the same column but different bands or which are in the same row but different stacks. If $\gcd(n,r)=1$ then $\alpha = n$ always creates a Latin square, and if $\gcd(n,c)=1$ then $\beta = n$ always creates a Latin square. 

To exemplify Algorithm \ref{generalized algorithm}, let $n=9$, and suppose we want to create a square of order 9 with distance 4, the theoretical maximum. By Theorem \ref{coprime increment theorem}, the set $K = \{1 + 4m \pmod{9} \mid m\in G\} = \{1,2,\dots, n\}$. If we list the elements in $K$ in the natural order given by the expression $1+4m\pmod{9}$, we obtain $(1, 5, 9, 4, 8, 3, 7, 2, 6)$. Note that each ``adjacent'' coordinate in this vector has a difference of 4, modulo 9. Therefore we can fill the top row of an $n\times n$ grid with this list to make it Latin. We can do the same along each column by adding 4 along the way down as well. The resulting matrix is a Latin square with distance 4, in fact this one is a back circulant Latin square (if we had added $-4$ vertically, we obtain the circulant Latin square seen in Figure \ref{fig:LS9}). This works nicely since $\gcd(9,4) = 1$, i.e. the row and column increments are co-prime to $n$. 

When the row or column increment is not co-prime to $n$, we offset every stack or band to ensure the row or column stays Latin. For example if $n=9$ and $c = 3$, we would obtain $(1, 4, 7, 1, 4, 7, 1, 4, 7)$. But if every third cell (new stack) we added an extra one, we arrive at $(1, 4, 7, 2, 5, 8, 3, 6, 9)$, a row with distance 3 everywhere except the boundaries of the stacks, where it has distance 4. 

Algorithm \ref{generalized algorithm} is a way to create a Latin square with specific \innerDist. In some cases, it is the only way to produce Latin squares of maximum \innerDist~(see Theorem \ref{number of odd MmD squares}), but there are also cases where it fails to produce the best results (see Section \ref{Sudoku section}).

\begin{corollary}\label{InnerDist of algo}
The  \innerDist~of a Latin square from Algorithm \ref{generalized algorithm} is $\mathrm{min}\{\pm r \pmod{n},$ $\pm c \pmod{n},$ $\pm (r+\alpha) \pmod{n},$ $\pm (c+\beta) \pmod{n}\}$.
\end{corollary}

The \AdjDist~of any given cells using Algorithm \ref{generalized algorithm} is $\pm r \pmod{n}$, $\pm c \pmod{n}$, or if crossing to a cell in a different block, $\pm (r+\alpha) \pmod{n}$, or $\pm (c+\beta) \pmod{n}$. The \innerDist~ is the minimum of these values.

\begin{theorem}
\label{MainResult}
If $n\geq3$,  then $$\MmD(L_n)= \left\lfloor\frac{n-1}{2}\right\rfloor.$$
\end{theorem}
\begin{proof}
Recall, by Lemma \ref{upperbound}, $\MmD(L_n)\leq \left\lfloor\frac{n-1}{2}\right\rfloor.$  We now construct Latin squares with  \innerDist~$\left\lfloor\frac{n-1}{2}\right\rfloor$ for any order $n$ using Algorithm \ref{generalized algorithm} . 

For odd $n$, using Algorithm \ref{generalized algorithm}, let $c=r=\frac{n-1}{2}$. By Theorem \ref{gcd stuff}, $\gcd(n,c) = \gcd(n,r) = 1$, and so let $\alpha=\beta=n$. Then, Algorithm \ref{generalized algorithm} creates a Latin square with \innerDist~$\frac{n-1}{2}$, the minimum of horizontal and vertical increments. 

For even $n$, we do not always have $\gcd(\frac{n-2}{2},n)=1$, but since $\frac{n}{2}$ is a distance larger than our bound of $\frac{n-2}{2}$, we can use $\frac{n}{2}$ for our increments along the bands and stacks. Therefore, using Algorithm \ref{generalized algorithm}, let $c = r = \frac{n-2}{2}$ along the columns with $\alpha = \beta=1$, then the minimum \innerDist~will be $\frac{n-2}{2}$, as all distances are either $\frac{n-2}{2}$ or $\frac{n}{2}$.
\end{proof}
\subsection{Examples of Squares with Maximum Inner Distance}

\begin{figure} [H]
    \centering
    \begin{tikzpicture}[scale=.6]
\draw(0,0)grid(6,6); 
\draw[step=6,ultra thick](0,0)grid(6,6);
\foreach\x[count=\i] in{1, 3, 5, 2, 4, 6}{\node at(\i-0.5,5.5){$\x$};};
\foreach\x[count=\i] in{5, 1, 3, 6, 2, 4}{\node at(\i-0.5,4.5){$\x$};};
\foreach\x[count=\i] in{3, 5, 1, 4, 6, 2}{\node at(\i-0.5,3.5){$\x$};};
\foreach\x[count=\i] in{6, 2, 4, 1, 3, 5}{\node at(\i-0.5,2.5){$\x$};};
\foreach\x[count=\i] in{4, 6, 2, 5, 1, 3}{\node at(\i-0.5,1.5){$\x$};};
\foreach\x[count=\i] in{2, 4, 6, 3, 5, 1}{\node at(\i-0.5,0.5){$\x$};};
\end{tikzpicture}
    \caption{$6 \times 6$ Non-circulant Latin square with inner distance of 2 with variables in Algorithm \ref{generalized algorithm} being $r=4, c=2, R=3, C=3, \alpha=-1$, and $\beta =1$.}
    \label{fig:LS6}
\end{figure}

\begin{figure} [H]
    \centering
    \begin{tikzpicture}[scale=.6]
\draw(0,0)grid(9,9);
\draw[step=9,ultra thick](0,0)grid(9,9);
\foreach\x[count=\i] in{ 1, 5, 9, 4, 8, 3, 7, 2, 6}{\node at(\i-0.5,8.5){$\x$};};
\foreach\x[count=\i] in{6, 1, 5, 9, 4, 8, 3, 7, 2}{\node at(\i-0.5,7.5){$\x$};};
\foreach\x[count=\i] in{2, 6, 1, 5, 9, 4, 8, 3, 7}{\node at(\i-0.5,6.5){$\x$};};
\foreach\x[count=\i] in{7, 2, 6, 1, 5, 9, 4, 8, 3}{\node at(\i-0.5,5.5){$\x$};};
\foreach\x[count=\i] in{3, 7, 2, 6, 1, 5, 9, 4, 8}{\node at(\i-0.5,4.5){$\x$};};
\foreach\x[count=\i] in{8, 3, 7, 2, 6, 1, 5, 9, 4}{\node at(\i-0.5,3.5){$\x$};};
\foreach\x[count=\i] in{4, 8, 3, 7, 2, 6, 1, 5, 9}{\node at(\i-0.5,2.5){$\x$};};
\foreach\x[count=\i] in{9, 4, 8, 3, 7, 2, 6, 1, 5}{\node at(\i-0.5,1.5){$\x$};};
\foreach\x[count=\i] in{5, 9, 4, 8, 3, 7, 2, 6, 1}{\node at(\i-0.5,0.5){$\x$};};
\end{tikzpicture}
    \caption{By Algorithm \ref{generalized algorithm} with $r=5, c=4,$ so that $R=9, C=9, \alpha=9$, and $\beta =9$, this is a $9 \times 9$ Latin square with inner distance 4.}
    \label{fig:LS9}
\end{figure}

\begin{figure} [H]
    \centering
    \begin{tikzpicture}[scale=.6]
\draw(0,0)grid(10,10);
\draw[step=10,ultra thick](0,0)grid(10,10);
\foreach\x[count=\i] in{1, 5, 9, 3, 7,  2,  6,  10,  4,  8}{\node at(\i-0.5,9.5){$\x$};};
\foreach\x[count=\i] in{6,  10,  4,  8,  2,  7,  1, 5, 9, 3}{\node at(\i-0.5,8.5){$\x$};};
\foreach\x[count=\i] in{2, 6, 10, 4, 8, 3, 7, 1, 5, 9}{\node at(\i-0.5,7.5){$\x$};};
\foreach\x[count=\i] in{7, 1, 5, 9, 3, 8, 2, 6, 10, 4}{\node at(\i-0.5,6.5){$\x$};};
\foreach\x[count=\i] in{3, 7, 1, 5, 9, 4, 8, 2, 6, 10}{\node at(\i-0.5,5.5){$\x$};};
\foreach\x[count=\i] in{8, 2, 6, 10, 4, 9, 3, 7, 1, 5}{\node at(\i-0.5,4.5){$\x$};};
\foreach\x[count=\i] in{4, 8, 2, 6, 10, 5, 9, 3, 7, 1}{\node at(\i-0.5,3.5){$\x$};};
\foreach\x[count=\i] in{9, 3, 7, 1, 5, 10, 4, 8, 2, 6}{\node at(\i-0.5,2.5){$\x$};};
\foreach\x[count=\i] in{5, 9, 3, 7, 1, 6, 10, 4, 8, 2}{\node at(\i-0.5,1.5){$\x$};};
\foreach\x[count=\i] in{10, 4, 8, 2, 6, 1, 5, 9, 3, 7}{\node at(\i-0.5,0.5){$\x$};};
\end{tikzpicture}
    \caption{By Algorithm \ref{generalized algorithm} with $r=5, c=4, R=2, C=5, \alpha=1$, and $\beta =1$, $10 \times 10$ Latin square with inner distance 4.}
    \label{fig:LS10}
\end{figure}

\subsection{Intermediate Results and Applications}
Two Latin squares $M$ and $N$ are \textbf{isotopic} if a sequence of row, column, and symbol permutations can transform $M$ into $N$. The \textbf{isotopy class} of a Latin square $N$ of order $n$ is the set of Latin squares $M$ of order $n$ that are isotopic to $N$.

\begin{lemma}\label{shift by k isotopic}
All shift-by-$k$ squares of order $n$ are isotopic.
\end{lemma}
\begin{proof}
The isotopy class of a Latin square is invariant under symbol permutation, so we may assume the top row is $(1,2,\dots, n)$. A shift-by-$k$ square takes the first row and shifts the values down 1 cell and to the right by $k$ cells, hence the order of each row is the same, only starting at a different value. Then for each $i$, place the row that contains a 1 in the $i^{\rm th}$ column into the $i^{\rm th}$ row. The resulting array is a circulant Latin square with first row $(1,\dots, n)$. We have shown that any shift-by-$k$ Latin square of order $n$ can be symbol and row permuted into this square, therefore they belong to the same isotopic class.
\end{proof}

\begin{lemma}\label{odd squares lemma}
There are two ways to arrange the first row and two ways to arrange the first column, up to symbol permutation, of an odd ordered Latin square that achieves the \MaxMinDist.
\end{lemma}
\begin{proof}
Let $n$ be odd. For any given $x \in \{1,2,\dots, n\}$, if the square has the maximum \innerDist, then there are two integers that can be adjacent to $x$, namely $x\pm \frac{n-1}{2}$ (see the discussion in the proof of Lemma \ref{upperbound}).

Assume we have an ordered set $(1,u_2,\dots, u_n)$. This may represent the first row or first column of a Latin square. Note that we only have two choices for $u_2,$ namely the integers $1\pm \frac{n-1}{2}$. We will show that choosing $u_2 = 1+\frac{n-1}{2}$ fills the rest of the square uniquely. The second case is identical.

Let $u_2= 1+ \frac{n-1}{2}$. Then $u_3$ has a unique symbol that has not been used, namely $(1 + \frac{n-1}{2})+\frac{n-1}{2}$.

Assume that for all $2\leq i\leq k-1$, $u_{i}=u_{i-1}+\frac{n-1}{2}$ and that for each $i$, $u_{i}=u_{i-1}+\frac{n-1}{2}$ was the only choice for $u_i$.
Since $u_{k-1}=u_{k-2}+\frac{n-1}{2}$, then ${u_k}$ cannot be $u_{k-1}-\frac{n-1}{2}$, as $u_{k-2}$ takes this value. Hence $u_k=u_{k-1}+\frac{n-1}{2}$ is the only option for $u_k$.

Because $\gcd(\frac{n-1}{2},n)=1$ for odd $n$, this ordered set is Latin, by Theorem \ref{modular divides stuff}. Therefore if $u_2=1 + \frac{n-1}{2}$, the rest of the list is determined by adding $\frac{n-1}{2}$ along the row, and similarly the list is determined by subtracting this amount when $u_2 = 1-\frac{n-1}{2}$. Then, up to symbol permutation of the first symbol, there are 2 ways to fill a row or column of an odd ordered square such that the Latin square has \MaxMinDist.
\end{proof}

\begin{theorem}\label{number of odd MmD squares}
For odd $n$, there are only $4n$ Latin squares with \innerDist~$\frac{n-1}{2}$. These squares are all circulant and thus in the same isotopic class.
\end{theorem}
\begin{proof}
For $n=3$, there are 12 Latin squares using the symbols $\{1, 2,3\}$, all with distance 1, since this is the theoretical maximum and minimum. Therefore the theorem trivially holds true for $n=3$. Assume $n>3$.

We have $n$ choices for the top left cell. Assume without loss of generality $m_{1,1}=1$. 

We have two ways to fill row 1 and column 1. We claim that if we add or subtract $\frac{n-1}{2}$ along the first row or column, then we must do the same along every row or column.
Consider the top $3\times 3$ region of the resulting array (recall $n\geq 3$ for this distance to hold). Let $\mu$ denote the \MaxMinDist~of an odd order $n$ Latin square, that is $\mu = \frac{n-1}{2}$. Note then that $2\mu \equiv -1$ and $3\mu \equiv \frac{n-3}{2} \pmod{n}$.

\begin{figure}[H]
    \centering
        \begin{tikzpicture}[scale=1.1]
\draw(0,0)grid(4,4); 
\draw[step=4,ultra thick](0,0)grid(4,4);
\foreach\x[count=\i] in{1, 1 + \mu, 1 + 2\mu, \dots}{\node at(\i-0.5,3.5){$\x$};};
\foreach\x[count=\i] in{1+\mu, x, y, \dots}{\node at(\i-0.5,2.5){$\x$};};
\foreach\x[count=\i] in{1+2\mu, z, \ddots }{\node at(\i-0.5,1.5){$\x$};};
\foreach\x[count=\i] in{\vdots, \vdots }{\node at(\i-0.5,0.5){$\x$};};
\end{tikzpicture}
    \caption{The top $3\times 3$ region of the square. Note that $x$ determines the values of the entire Latin square.}
    \label{2 by 3 region}
\end{figure}

Here we choose to add $\mu$ along rows and columns, the other 3 cases are identical. To maintain the \MaxMinDist~of $\mu$, this increment must be constant along an entire first row and column, by Lemma \ref{odd squares lemma}, hence why the values in cells $(3,1)$ and $(1,3)$ are pre-determined. 

Note that cell $(2,2)$ has two possibilities: $1+\mu \pm\mu$, since these are the two neighbors that $1+ \mu$ can have with \MaxMinDist.

If $x = 1$, then $y = 1-\mu$ by Lemma \ref{odd squares lemma}. Then $m_{1,3} - m_{2,3} = (1+2\mu) - (1-\mu) = 3\mu \equiv \frac{n-3}{2} < \mu$. Therefore, this square would not have \MaxMinDist.
Therefore $x = 1+2\mu$, and so we increment by $+\mu$ along rows 1 and 2. Since we incremented by $\mu$ vertically along every column, the rest of the square must be filled in this manner as well. Then we have constructed $n$ Latin squares with \MaxMinDist, pertaining to our choice of symbol in the top left. In a similar fashion, we may have chosen to increment by $- \mu$ vertically or horizontally, summing up to $4n$ total squares of \MaxMinDist. 
\end{proof}

The following theorem suggests that most Latin squares with \MaxMinDist~are cyclic in nature. 

\begin{theorem}
Any Latin square generated by Algorithm \ref{generalized algorithm} is isotopic to a circulant Latin square.
\end{theorem}
\begin{proof}
Suppose without loss of generality that $\gcd(n,c) = 1$. Then row $i$ takes values $(m_{i, 1}, m_{i,1}+c, \dots, m_{i,1}+(n-1)c)$ (with entries modulo $n$). So every row is in the same order, beginning with a different integer. Therefore if we permute the rows such that the leftmost entry in row $i$ is $m_{1, 1} + (i-1)c$, the resulting square is a back-circulant Latin square. To see this, note that for some $m_{i,j}$ with say $i>1$ and $j < n$, $m_{i,j} = m_{i,1} + (j-1)c = m_{i-1, 1} + jc$, and similarly $m_{i-1, j+1} = m_{i-1, 1} + jc$. Here we use the fact that we increment by $c$ horizontally and along the first column after we permute the rows. So $m_{i,j} = m_{i-1, j+1}$, the exact condition for back-circulant Latin squares. If $i = 1$ we can replace $i-1$ with $n$ and likewise if $j = n$ replace $j+1$ with 1. 

The proof is similar if $\gcd(n,r)=1$, so if either $\gcd(n,r)$ or $\gcd(n,c)$ is 1 then the theorem holds. We now work with the case where $\gcd(n,r)\neq 1$ and $\gcd(n,c) \neq 1$. 

Permute the columns so that $m_{1,j} = j$, so that the top row in order is $(1,2,\dots, n)$. Since the order of entries in a column are left un-modified by column permutations, $m_{i, j} = m_{1,j} + (i-1)r + B_i$, where $B_i$ is the band $m_{i,j}$ lands on (here we take $\alpha = \beta = 1$, otherwise we replace $B_i$ with the expression in Algorithm \ref{generalized algorithm}). Therefore $m_{i,j} = j + (i-1)r + B_i$, so row $i$ is, in order, $(m_{i, 1}, m_{i, 1} + 1, \dots, n, 1, \dots, m_{i,1} - 1)$. Permute the rows so that cell $m_{i, 1} = i$, and the resulting square is a circulant Latin square, in fact it is the circulant Latin square with first column and first row being the integers $\{1,2,\dots, n\}$ in order.
\end{proof}

In particular, all $4n$ squares of maximum inner distance for odd $n$ are created by Algorithm \ref{generalized algorithm}, so they are isotopic.

We now shift our focus to other variants of Latin squares. The results in this section can be applied as an upper bound for any other type of Latin squares, though adding certain restrictions such as the rules of Sudoku increases the difficulty both in creating squares and proving a \MaxMinDist.

\section{Pandiagonal Latin Squares}
Throughout the paper we discussed having each integer exactly once in every column and row. Now we briefly consider diagonals. A \term{back diagonal} is a set of cells $(i,j)$ such that $i+j \equiv d_b \pmod n $ is a constant. The set defined by $d_b = n+1$ is sometimes called the minor diagonal. A \term{forward diagonal} is a set of cells with $i-j \equiv d_f \pmod n$ as a constant. The set defined by $d_f = 0$ is sometimes called the major or main diagonal. We will use the term \textit{diagonal} to speak about both forward and back diagonals.

A \term{pandiagonal Latin square} $P$ is a Latin square with order $n$ where every diagonal is Latin. That is, the set of symbols on any forward or back diagonal is $\{1,2,\dots, n\}$. It is worth noting that these diagonals are an example of a transversal, which is a set $T$ of $n$ cells such that each column and row contains one cell from $T$, and the symbol set of $T$ is $\{1,2,\dots, n\}$. We use $P_n$ to denote the set of all pandiagonal Latin squares of order $n$.

The topic of existence of pandiagonal Latin squares was studied in \cite{knut_vik_designs}, and it was shown that a pandiagonal Latin square of order $n$ exists if and only if $n \equiv 1, 5 \pmod{6}$. This is equivalent to stating that 2 and 3 do not divide $n$. Pandiagonal Latin squares are also known as Knut-Vik designs, as in \cite{knut_vik_designs}.

\begin{theorem}
If $n \equiv 1,5 \pmod{6}$, then $\MmD(P_n)=\frac{n-3}{2}$.
\end{theorem}
\begin{proof}
Let $P\in P_n$ be a pandiagonal Latin square of order $n$, for $n \equiv 1,5 \pmod{6}$. Since $n$ is odd, $P$ has inner distance at most $\frac{n-1}{2}$. From the proof of Theorem \ref{number of odd MmD squares}, the only way to construct a distance $\frac{n-1}{2}$ square is to use a circulant Latin square, which using Algorithm \ref{generalized algorithm} increments by $\pm \frac{n-1}{2}$ along columns and rows, i.e. $c\in\{\pm \frac{n-1}{2}\}$ and $r\in\{\pm \frac{n-1}{2}\}$. Now consider the main forward diagonal, with cells $(i,i)$. Then, for $1\leq i \leq n-1$, $(m_{i+1,i+1}-m_{i,i})\equiv c+r \pmod n$. Similarly, for the main back diagonal with cells $(i,n-i+1)$, we have for $1\leq i \leq n-1$, $(m_{i+1,n-(i+1)+1}-m_{i,n-i+1})\equiv c-r \pmod n$.
Then since $|c|=|r|$ for these squares, we will have that either the main forward or main back diagonal  contains just one symbol. Therefore there is no pandiagonal with \innerDist~$\frac{n-1}{2}$.

Using Algorithm \ref{generalized algorithm}, we can construct a pandiagonal Latin square with inner distance $\frac{n-3}{2}$, using $r= \frac{n-1}{2}$, $c=- \frac{n-3}{2}$, $\alpha=n$, and $\beta=n$. See Figure \ref{fig:P_11} for an example. Then, on the main forward diagonal, we have $(m_{i+1,i+1}-m_{i,i})\equiv 1 \pmod n$. Similarly, using the same orientation as on the main forward diagonal, the difference between entries for cells which are adjacent on any forward diagonal will be $1$. On the main back diagonal, we have $(m_{i+1,n-(i+1)+1}-m_{i,n-i+1})\equiv -2 \pmod n$. Similarly, using the same orientation as on the main back diagonal, the difference between entries for cells which are adjacent on any back diagonal will be $- 2$. Since $n$ is odd, we can see that the diagonals are always Latin by Theorem \ref{coprime increment theorem}, and thus the resulting Latin square is pandiagonal. By Corollary \ref{InnerDist of algo}, the distance of this pandiagonal Latin square is $\frac{n-3}{2}$. 
\end{proof}

\begin{figure} [H]
    \centering
    \begin{tikzpicture}[scale=.6]
\draw(0,0)grid(11,11);
\draw[step=11,ultra thick](0,0)grid(11,11);

\foreach\x[count=\i] in{1, 6, 11, 5, 10, 4, 9, 3, 8, 2, 7}{\node at(\i-0.5,10.5){$\x$};};
\foreach\x[count=\i] in{8, 2, 7, 1, 6, 11, 5, 10, 4, 9, 3}{\node at(\i-0.5,9.5){$\x$};};
\foreach\x[count=\i] in{4, 9, 3, 8, 2, 7, 1, 6, 11, 5, 10}{\node at(\i-0.5,8.5){$\x$};};
\foreach\x[count=\i] in{11, 5, 10, 4, 9, 3, 8, 2, 7, 1, 6}{\node at(\i-0.5,7.5){$\x$};};
\foreach\x[count=\i] in{7, 1, 6, 11, 5, 10, 4, 9, 3, 8, 2}{\node at(\i-0.5,6.5){$\x$};};
\foreach\x[count=\i] in{3, 8, 2, 7, 1, 6, 11, 5, 10, 4, 9}{\node at(\i-0.5,5.5){$\x$};};
\foreach\x[count=\i] in{10, 4, 9, 3, 8, 2, 7, 1, 6, 11, 5}{\node at(\i-0.5,4.5){$\x$};};
\foreach\x[count=\i] in{6, 11, 5, 10, 4, 9, 3, 8, 2, 7, 1}{\node at(\i-0.5,3.5){$\x$};};
\foreach\x[count=\i] in{2, 7, 1, 6, 11, 5, 10, 4, 9, 3, 8}{\node at(\i-0.5,2.5){$\x$};};
\foreach\x[count=\i] in{9, 3, 8, 2, 7, 1, 6, 11, 5, 10, 4}{\node at(\i-0.5,1.5){$\x$};};
\foreach\x[count=\i] in{5, 10, 4, 9, 3, 8, 2, 7, 1, 6, 11}{\node at(\i-0.5,0.5){$\x$};};
\end{tikzpicture}
    \caption{$11 \times 11$ pandiagonal Latin square with $\MmD(P_{11})= 4$.}
    \label{fig:P_11}
\end{figure}

\section{Sudoku Latin Squares}\label{Sudoku section}
Sudoku Latin squares are a popular type of Latin square and an active area of research  due to the popular game \textit{Sudoku}. The typical Sudoku puzzle that players must complete can be referred to as a $(3,3)$-Sudoku Latin square because players fill a $9 \times 9$ array into a Latin square with the added constraint that each $3\times 3$ \textit{block} also contains the integers 1 to 9. This now popular puzzle began in 1975 as a puzzle by Howard Garns which appeared in Dell Magazine called ``Number Places''\cite{Garns}. 

In general, for positive integers $a$ and $b,$ an \term{$(a,b)$-Sudoku Latin square} is a Latin square with order $n=ab$ where the $n \times n$ Latin square is tiled with $a \times b$ blocks (the intersections of bands of height $a$ and stacks of width $b$) such that each $a \times b$ block contains all $n$ symbols in $\{1,2,\dots, n\}$ exactly once \cite{book:Combinatorics}. We investigate the \MaxMinDist~of $(a,b)$-Sudoku Latin squares and find some general results as well as results for specific cases dependent on $a$ and $b$. As we shall see, the \MaxMinDist~depends now on $(a,b)$ rather than $n$, so we denote the maximum inner distance over all $(a,b)$-Sudoku Latin squares as $\MmD(L_{a,b})$.

\begin{lemma}
For any positive integers $a$ and $b, \MmD(L_{a,b})=\MmD(L_{b,a})$.
\end{lemma}
\begin{proof}
Let $L$ be an $(a,b)$-Sudoku Latin square. Then $L^T$ is a $(b,a)$-Sudoku Latin square, where adjacent symbols, and hence distances, are preserved.
\end{proof}

Because of this, for the remainder of the paper we work with the convention that $a\leq b$. 

\begin{theorem}\label{(2,b) case}
If $b\geq 2$, $\MmD(L_{2,b})=b-1 = \frac{n-2}{2}$.
\end{theorem}
\begin{proof}
By Theorem \ref{MainResult}, $\MmD(L_{2,b}) \leq \left\lfloor\frac{2b-1}{2}\right\rfloor=\left\lfloor\frac{n-1}{2}\right\rfloor$. Since $n=2b$ is even,  $\left\lfloor\frac{n-1}{2}\right\rfloor=\frac{n-2}{2}$, and so $\MmD(L_{2,b})\leq\frac{n-2}{2}$.

We will now use Algorithm \ref{generalized algorithm} to construct a $2 \times b$ Latin square with inner distance $\frac{n-2}{2}$.

\textit{Case 1}: Let $b$ be odd.  In Algorithm \ref{generalized algorithm}, let $r = \frac{n}{2}, \alpha = 1, c = \frac{n-2}{2},$ and $\beta = 1$. The resulting square is a Latin square with inner distance $\frac{n-2}{2}$. It remains to show that the Latin square produced is an $(2,b)$-Sudoku Latin square.  

From Algorithm \ref{generalized algorithm}, the entry for cell $(i,j)$ would be 
\[m_{i,j} = 1 + (i-1)\frac{n}{2} + (j-1)\frac{n-2}{2} +\left\lfloor{\frac{i-1}{R}}\right\rfloor + \left\lfloor{\frac{j-1}{C}}\right\rfloor \pmod{n}.\]

Now suppose two entries, $m_{i_1,j_1}$ and $m_{i_2,j_2}$ in the same $2\times b$ block are the same. Since $m_{i_1,j_1}$ and $m_{i_2,j_2}$ are in the same block, they are in the same band, and $\left\lfloor{\frac{i_1-1}{R}}\right\rfloor=\left\lfloor{\frac{i_2-1}{R}}\right\rfloor$. Likewise, $m_{i_1,j_1}$ and $m_{i_2,j_2}$ are in the same stack, and so $\left\lfloor{\frac{j_1-1}{C}}\right\rfloor=\left\lfloor{\frac{j_2-1}{C}}\right\rfloor$. Setting the expressions for $m_{i_1,j_1}$ and $m_{i_2,j_2}$ equal and cancelling terms, we have:

\[ \frac{n}{2}i_1 +  \left(\frac{n-2}{2}\right)j_1 \equiv \frac{n}{2}i_2 +  \left(\frac{n-2}{2}\right)j_2\pmod{n}. \]

Rewriting in terms of $b$, we have 
\[ bi_1 +  (b-1)j_1 \equiv bi_2 +  (b-1)j_2\pmod{n}\iff b(i_1 - i_2) \equiv (b-1)(j_2-j_1)\pmod{n}. \]

By Theorem \ref{modular divides stuff}, $\gcd(b,n)\vert (b-1)(j_2-j_1)$. Since $\gcd(b,n)=\gcd(b,2b)=b$ and $\gcd(b, b-1)=1$, $b \vert (j_2-j_1)$. Since $m_{i_1,j_1}$ and $m_{i_2,j_2}$ are in the same block, then $0\leq |j_2-j_1| \leq b-1$. Therefore $j_2 = j_1$, and we now have that \[ b(i_1 - i_2) \equiv 0 \pmod{n}. \] Since $n=2b$, $2\vert (i_1 - i_2)$. However, since $m_{i_1,j_1}$ and $m_{i_2,j_2}$ are in the same block, $0\leq |i_2-i_1| \leq 1$ and so, $i_2 = i_1$. 

\textit{Case 2}: Let $b$ be even. Then use Algorithm \ref{generalized algorithm} with $r=\frac{n}{2}, \alpha = 1, c = \frac{n-2}{2},$ and $\beta = n$.
By a similar reasoning as Case 1, this produces a $(2,b)$-Sudoku Latin square with inner distance $b-1 = \frac{n-2}{2}$.
\end{proof}

\begin{lemma}\label{theoretical upper bound}
If $a, b \geq 3 $, then $\MmD(L_{a,b}) \leq \left \lfloor\frac{n-3}{2}\right \rfloor$.
\end{lemma}
\begin{proof}
If $a,b \geq 3$, then there is at least one $3\times3$ region inside any $a\times b$ block. Let $u$ be in the center cell of a $3\times3$ region as seen in Figure \ref{fig:CenterThreeByThree}.
\begin{figure}[H]
\centering
\begin{tikzpicture}[scale=.6]
\draw(0,0)grid(3,3); 
\draw[step=3,ultra thick](0,0)grid(3,3);
\foreach\x[count=\i] in{, i, }{\node at(\i-0.5,2.5){$\x$};};
\foreach\x[count=\i] in{j, u, k}{\node at(\i-0.5,1.5){$\x$};};
\foreach\x[count=\i] in{ , \ell, }{\node at(\i-0.5,0.5){$\x$};};
\end{tikzpicture}
\caption{A $3 \times 3$ region contained in an $a\times b$ block.}
\label{fig:CenterThreeByThree}
\end{figure}
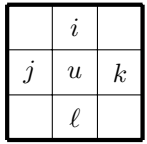

Since an $a \times b$ block contains all $n=ab$ symbols exactly once, $u, i$, $j$, $k$ and $\ell$ are all distinct. 

Suppose $n$ is odd. From the proof of Lemma \ref{upperbound}, there are only two integers with distance $\frac{n-1}{2}$ from $u$, the integers $u\pm \frac{n-1}{2}$. Hence at least two of $i,j,k,$ or $\ell$ must have distance less than $\frac{n-1}{2}$ from $u$, and so $\MmD(L_{a,b})\leq\frac{n-3}{2}$.

Suppose $n$ is even. From the proof of Lemma \ref{upperbound}, there is only one integer at maximum distance away from $u$, the integer $u+\frac{n}{2}$. Then at least three of $i,j,k,$ or $\ell$ has distance less than or equal to $\frac{n-2}{2}$. There are only two integers with distance $\frac{n-2}{2}$, the integers $u \pm \frac{n-2}{2}$. So, at least one of $i,j,k,$ or $\ell$ has distance less than $\frac{n-2}{2}$. Hence $\MmD(L_{a,b})\leq \frac{n-4}{2} = \floor{\frac{n-3}{2}}$.
\end{proof}

\begin{theorem}\label{a by odd}
Let $a \leq b$. If $b$ is odd, then $\MmD(L_{a,b}) \geq \frac{n-a}{2}$.
\end{theorem}
\begin{proof}

If $a=1$, and $b$ is odd, then $L_{1,b}=L_b$. By Theorem \ref{MainResult}, $\MmD(L_n)=\left\lfloor\frac{n-1}{2}\right\rfloor$. Since $n=b$ is odd, $\left\lfloor\frac{n-1}{2}\right\rfloor=\frac{n-1}{2}$, and the theorem holds.

If $a=2$, and $b$ is any positive integer, then the theorem holds by Theorem \ref{(2,b) case}.

Now suppose $a\geq 3$ and $b$ is an odd integer. Let $m_{i,j}$ be the entry on cell $(i,j)$ of $L \in L_{a,b}$. We use Algorithm \ref{generalized algorithm} with $c = \frac{n-a}{2}$, $\alpha = n$, and $\beta = 1$. The value for $r$ depends on $a$, but we choose $r$ such that $\frac{n-a}{2} < r < \frac{n}{2}$ and $\gcd(n,r) = 1$ as follows:

\textit{Case 1}: If $a$ is odd, then let $r= \frac{n-1}{2}$. Then $\gcd(n, r)=\gcd(n, \frac{n-1}{2}) =1$ by Theorem \ref{gcd stuff}.
   
\textit{Case 2}: If $a=4k$, then let $r = \frac{n-2}{2} = 2bk-1$. Since $\frac{n-2}{2}$ is odd, then $\gcd(n, \frac{n-2}{2})$ is also odd. Since any divisor of $\frac{n-2}{2}$ is a divisor of $n-2$, $\left.\gcd(n, \frac{n-2}{2})\right\vert \gcd(n,n-2)$. But since $\left.\gcd(n,n-2)\right\vert2$ and  $\gcd(n, \frac{n-2}{2})$ is odd, $\gcd(n, r)=\gcd(n, \frac{n-2}{2}) = 1$.
    
\textit{Case 3}: If $a=4k+2$, then let $r = \frac{n-4}{2}=\frac{b(4k+2)-4}{2}=2kb+b-2=2(kb-1)+b$. Since $b$ is odd, $\frac{n-4}{2}=2(kb-1)+b$ is odd, and so $\gcd(n, \frac{n-4}{2})$ is odd. Since any divisor of $\frac{n-4}{2}$ is a divisor of $n-4$, $\left.\gcd(n, \frac{n-4}{2})\right\vert \gcd(n,n-4)$. But since $\left.\gcd(n,n-4)\right\vert4$ and  $\gcd(n, \frac{n-4}{2})$ is odd, $\gcd(n, r)=\gcd(n, \frac{n-4}{2}) = 1$.

$L$ is a Latin square, and by Corollary \ref{InnerDist of algo}, $L$ has an inner distance $\frac{n-a}{2}$ because $\frac{n-a}{2}$ is the minimum of adjacent distances. We show that $L$ is an $(a,b)$-Sudoku Latin square by proving that each of its $a\times b$ blocks are Latin.

From Algorithm \ref{generalized algorithm}, the entry for cell $(i,j)$ would be 
\[m_{i,j} = 1 + (i-1)r + (j-1)\frac{n-a}{2} +n\left\lfloor{\frac{i-1}{R}}\right\rfloor + \left\lfloor{\frac{j-1}{C}}\right\rfloor \pmod{n}\] or equivalently

\[m_{i,j} = 1 + (i-1)r + (j-1)\frac{n-a}{2} + \left\lfloor{\frac{j-1}{C}}\right\rfloor \pmod{n}.\]

Now suppose two entries, $m_{i_1,j_1}$ and $m_{i_2,j_2}$ in the same $a\times b$ block are the same. Since $m_{i_1,j_1}$ and $m_{i_2,j_2}$ are in the same block, they are in the same stack, and so $\left\lfloor{\frac{j_1-1}{C}}\right\rfloor=\left\lfloor{\frac{j_2-1}{C}}\right\rfloor$. Setting the expressions for $m_{i_1,j_1}$ and $m_{i_2,j_2}$ equal and cancelling terms, we have
\[
\frac{n-a}{2}(j_1 - j_2) \equiv r(i_2-i_1)\pmod{n}.
\]
Note that $\gcd(n, \frac{n-a}{2}) = a$ by Theorem \ref{gcd stuff}. By Theorem \ref{modular divides stuff}, $a$ divides $r(i_2-i_1)$. 

Since the $\gcd(n, r)=1$, then $a$ divides $i_2-i_1$. This is only possible when $i_1 = i_2$, since $m_{i_2}$ and $m_{i_1}$ are in the same block, they are in same band, and so $0\leq |i_2-i_1|\leq a-1$.  So we now have
\[\frac{n-a}{2}(j_1-j_2) \equiv 0 \pmod{n}.\]
Therefore, by Theorem \ref{modular divides stuff}, $\frac{n}{a} = b$ divides $j_1-j_2$. This is only possible if $j_1 = j_2$ since $j_2$ and $j_1$ are along the same block, with width $b$.  So no two distinct cells in an $a\times b$ block contain the same integer. 
\end{proof}

\begin{figure}[H]
\centering
\begin{tikzpicture}[scale=.6]
\draw(0,0)grid(9,9);
\draw[step=3,ultra thick](0,0)grid(9,9);
\foreach\x[count=\i] in{1,	6,	2,	7,	3,	8,	4,	9,	5}{\node at(\i-0.5,8.5){$\x$};};
\foreach\x[count=\i] in{7,	3,	8,	4,	9,	5,	1,	6,	2}{\node at(\i-0.5,7.5){$\x$};};
\foreach\x[count=\i] in{4,	9,	5,	1,	6,	2,	7,	3,	8}{\node at(\i-0.5,6.5){$\x$};};
\foreach\x[count=\i] in{9,	5,	1,	6,	2,	7,	3,	8,	4}{\node at(\i-0.5,5.5){$\x$};};
\foreach\x[count=\i] in{6,	2,	7,	3,	8,	4,	9,	5,	1}{\node at(\i-0.5,4.5){$\x$};};
\foreach\x[count=\i] in{3,	8,	4,	9,	5,	1,	6,	2,	7}{\node at(\i-0.5,3.5){$\x$};};
\foreach\x[count=\i] in{8,	4,	9,	5,	1,	6,	2,	7,	3}{\node at(\i-0.5,2.5){$\x$};};
\foreach\x[count=\i] in{5,	1,	6,	2,	7,	3,	8,	4,	9}{\node at(\i-0.5,1.5){$\x$};};
\foreach\x[count=\i] in{2,	7,	3,	8,	4,	9,	5,	1,	6}{\node at(\i-0.5,0.5){$\x$};};
\end{tikzpicture}
\caption{A $(3,3)$-Sudoku Latin Square with $\MmD(L_{3,3}) =3$}
\label{fig:(3,3)ex}
\end{figure}

Algorithm \ref{generalized algorithm} does not always yield a Latin square of maximum inner distance. The following algorithm is more complicated due to the fact that $\gcd(n,\frac{n-a}{2})=\frac{a}{2}$, so the horizontal offsets occur every 2 stacks ($2b$ cells) instead. This affects how integers are placed within a block, and requires a modification of Algorithm \ref{generalized algorithm}.

\begin{algorithm} \label{algo:EvenEven}
Let $a=2x,b=2y$, with $2\leq x\leq y\in \mathbb{N}$. Let $n = ab = 4xy$.\\
Fill cell $(i,j)$ of an $n\times n$ grid with the integer:
\[m_{i,j} = 1 + (j-1)(2xy-x) + \floor{\frac{j-1}{4y}} + (i-1)(2xy) + \sum\limits_{k=1}^{i} R(k)\]
or equivalently:
\[m_{i,j} = 1 + (j-1)\left(\frac{n-a}{2}\right) + \floor{\frac{j-1}{2b}} + (i-1)\frac{n}{2} + \sum\limits_{k=1}^{i} R(k)\]
\begin{center}
Where $R(k) = 
\begin{cases}
  0, & \text{if $k \equiv 0 \pmod{2}$ or if $k=1$}\\
  -x = -\frac{a}{2}, & \text{if $k \equiv 1 \pmod{a}$ and $k\neq 1$}\\
  1, & \text{if $1 < k \pmod{2a} < a$}\\
  -1, & \text{if $a+1 < k \pmod{2a} < 2a$}
\end{cases}.$
\end{center}
\end{algorithm}

\begin{example}
The simplest possible example is the $(4,4)$-Sudoku Latin Square where $x=y=2$. Using the algorithm, we horizontally increment by $2xy-x=6$, and by $2xy-x+1=7$ across the center. Vertically, for the entry $m_{i,j}$, if $i$ is even,  $m_{i,j}=m_{i-1,j}+8\pmod{16}$. \\
For odd $i>1$:\\
If $i\equiv 1\pmod{4},\,m_{i,j}=m_{i-1,j}+6 \pmod{16}$.\\
If $1<i\pmod{8}<4,\,m_{i,j}=m_{i-1,j}+9 \pmod{16}$.\\
If $5<i\pmod{8}<8,\,m_{i,j}=m_{i-1,j}+7 \pmod{16}$.
\end{example}
\begin{figure}[H]
    \centering
\tikzset{mystripes dist/.initial=0.25}
\pgfdeclarepatternformonly[/tikz/mystripes dist]{mystripes}
{\pgfpointorigin}{\pgfpoint{1cm}{1cm}}
{\pgfpoint{1cm}{1cm}}
{
  \foreach \x in {0,\pgfkeysvalueof{/tikz/mystripes dist},...,1}{
    \pgfmathsetmacro{\nx}{1-\x}
    \tikz[overlay]\draw[decoration = {random steps, segment length = 0.5mm, amplitude = 0.15pt}, color=black!40, decorate] (\x, 0) -- ++(\nx,\nx);
    \tikz[overlay]\draw[decoration = {random steps, segment length = 0.5mm, amplitude = 0.15pt}, color=black!40, decorate] (0, \x) -- ++(\nx,\nx);
  }
}
\begin{tikzpicture}[scale=.7]
\draw (-0.5,14) node[left]{$B_0$};
\draw (-0.5,10) node[left]{$B_1$};
\draw (-0.5,6) node[left]{$B_2$};
\draw (-0.5,2) node[left]{$B_3$};
\draw[fill=black!40](0,12) rectangle (16,11);
\draw[fill=black!40](0,8) rectangle (16,7);
\draw[fill=black!40](0,4) rectangle (16,3);

\draw[pattern=mystripes](0,14) rectangle (16,13);
\draw[pattern=mystripes](0,6) rectangle (16,5);

\draw[pattern=crosshatch dots](0,10) rectangle (16,9);
\draw[pattern=crosshatch dots](0,2) rectangle (16,1);

\draw(0,0)grid(16,16);
\draw[step=4,ultra thick](0,0)grid(16,16);

\foreach\x[count=\i] in{3, 9 , 15 , 5 , 11 , 1 , 7 , 13 , 4 , 10 , 16 , 6 , 12 , 2 , 8 , 14}{\node at(\i-0.5,0.5){$\x$};};
\foreach\x[count=\i] in{11 , 1 , 7 , 13 , 3 , 9 , 15 , 5 ,12 , 2 , 8 , 14 , 4 , 10 , 16 , 6}{\node at(\i-0.5,1.5){$\x$};};
\foreach\x[count=\i] in{4 , 10 , 16, 6 , 12 , 2 , 8 ,14 ,5 ,11 ,1 , 7 , 13 , 3 , 9 , 15}{\node at(\i-0.5,2.5){$\x$};};
\foreach\x[count=\i] in{12 , 2 , 8 , 14 , 4 , 10 , 16 , 6 , 13 , 3 , 9 , 15 ,5 , 11 , 1 , 7}{\node at(\i-0.5,3.5){$\x$};};
\foreach\x[count=\i] in{6 , 12 , 2 , 8 , 14 , 4 , 10 , 16 , 7 , 13 , 3 , 9 , 15 , 5 , 11 , 1}{\node at(\i-0.5,4.5){$\x$};};
\foreach\x[count=\i] in{14 ,4 , 10 , 16 , 6 , 12 , 2 , 8 , 15 , 5 , 11 , 1, 7 , 13 , 3 , 9}{\node at(\i-0.5,5.5){$\x$};};
\foreach\x[count=\i] in{5, 11 , 1, 7 , 13 , 3 , 9 , 15 , 6 , 12 , 2 , 8 , 14 , 4 , 10 , 16}{\node at(\i-0.5,6.5){$\x$};};
\foreach\x[count=\i] in{13 , 3 , 9 , 15 , 5 , 11 , 1 , 7 , 14 , 4 , 10 ,16 , 6 , 12 , 2 , 8}{\node at(\i-0.5,7.5){$\x$};};
\foreach\x[count=\i] in{7 , 13 , 3 ,9 , 15 ,5 ,11 , 1 , 8 ,14, 4 , 10 ,16 ,6 ,12, 2}{\node at(\i-0.5,8.5){$\x$};};
\foreach\x[count=\i] in{15 ,5 , 11 , 1, 7, 13, 3 , 9, 16, 6, 12, 2, 8,14, 4,10}{\node at(\i-0.5,9.5){$\x$};};
\foreach\x[count=\i] in{8,14,4,10,16, 6 , 12 , 2 , 9, 15 ,5 ,11 , 1 , 7 , 13 , 3}{\node at(\i-0.5,10.5){$\x$};};
\foreach\x[count=\i] in{16, 6, 12, 2, 8, 14, 4, 10, 1, 7, 13, 3, 9, 15, 5, 11}{\node at(\i-0.5,11.5){$\x$};};
\foreach\x[count=\i] in{10, 16, 6, 12, 2, 8, 14, 4, 11, 1, 7, 13, 3, 9, 15, 5}{\node at(\i-0.5,12.5){$\x$};};
\foreach\x[count=\i] in{2,8, 14, 4, 10, 16, 6, 12, 3, 9, 15, 5, 11, 1, 7, 13}{\node at(\i-0.5,13.5){$\x$};};
\foreach\x[count=\i] in{9, 15,5, 11, 1, 7, 13, 3, 10, 16, 6, 12, 2, 8, 14, 4}{\node at(\i-0.5,14.5){$\x$};};
\foreach\x[count=\i] in{1, 7, 13, 3, 9, 15, 5, 11, 2, 8,14,4,10,16, 6, 12}{\node at(\i-0.5,15.5){$\x$};};
\end{tikzpicture}
    \caption{A $(4,4)$-Sudoku Latin Square with the \MaxMinDist~of 6. The shading of row $i$ corresponds to the value of $R(i)$ (explained in the discussion below regarding type of cells)}
    \label{4,4 with MmD}
\end{figure}
The above algorithm can be expressed as follows:\\
Place a 1 in the top left corner of an $n\times n$ grid. Horizontally, increment (add to previous value) by $2xy-x$, adding 1 every $4y = 2b$ cells (or every two stacks). Vertically, increment by $2xy$ plus some constant $R(i)$, depending on the \textbf{type} of cell. There are 4 types, depending on the row $i$:

\begin{description}
    \item[Type 0)] The first row and all even rows has $R(i)=0$. In Figure \ref{4,4 with MmD}, Type 0 corresponds to cells without shading.
    \item[Type I)] The first row of every band except the first band ($B_0$) has $R(i) = -x$. In Figure \ref{4,4 with MmD}, Type I corresponds to solid gray cells.
    \item[Type II)] The odd rows in the even bands has $R(i) = 1$. In Figure \ref{4,4 with MmD}, Type II corresponds to cells with dashed lines.
    \item[Type III)] The odd rows in the odd bands has $R(i) = -1$. In Figure \ref{4,4 with MmD}, Type III corresponds to dotted cells.
\end{description}

Within every block except the first, the top row is the only set of Type I cells, the remaining odds are entirely Type II or entirely Type III, alternating vertically along the bands. All other cells are Type 0.

\begin{proof}[Proof The Algorithm Produces an $(a,b)$-Sudoku Latin Square]
Let $a=2x$, $b = 2y$ for integers $2\leq x \leq y$, and $n = ab = 4xy.$ To avoid fractions, we swap to using $2x,2y$ over $a, b$ (though we keep$\pmod{n}$ notation). It can be shown that $\frac{n-a}{2} = 2xy-x$, and $\gcd(n, \frac{n-a}{2}) = \gcd(4xy,2xy-x)=x = \frac{a}{2}$. We prove here the resulting array using the algorithm above is an $(a,b)$-Sudoku Latin square by showing the columns and $a\times b$ blocks are Latin. The proof that rows are Latin is identical to that of Algorithm \ref{generalized algorithm}.

{
\textit{Blocks}:
Suppose $m_{i_1, j_1} \equiv m_{i_2, j_2} \pmod{n}$. We set the two expressions for $m_{i,j}$ equal and arrive at the conclusion $i_1 = i_2$ and $j_1 = j_2$. Without loss of generality, if $i_1< i_2$:
\[m_{i_1, j_1} \equiv m_{i_2, j_2} \pmod{n}\iff (j_1-j_2)(2xy-x)- (i_2-i_1)(2xy)\equiv \sum_{k=i_1+1}^{i_2}R(k)\pmod{n}.\]
Here we cancelled constant terms including the floor term, assuming the cells are from the same block.

Note that $x$ divides the terms on the left hand side and $n=4xy$. Therefore by Theorem \ref{modular divides stuff}, $x$ divides the sum on the right. Within the range $[i_1+1, i_2]$, the values of $R(k)$ are all in $\{0, 1\}$, or all in $\{0, -1\}$, with at most $x-1$ non-zero terms (they appear only on odd rows). Thus since $x$ divides the sum:
\[-(x-1) \leq \sum_{k=i_1+1}^{i_2}R(k) \leq x-1\implies \sum_{k=i_1+1}^{i_2}R(k) = 0.\]
This can only be true if $i_2 = i_1$ or if $i_2 = i_1 + 1$ and $i_2$ is even. We assumed $i_1<i_2$:
\[(j_1-j_2)(2xy-x)- 2xy \equiv 0\pmod{n} \iff [(j_1-j_2)(2y-1)-2y]x\equiv 0\pmod{n}.\]
Since $m_{i_1, j_1}, m_{i_2, j_2}$ are in the same block, $0 \leq i_2-i_1 \leq 2x-1$ and $0 \leq |j_1-j_2| \leq 2y-1$. By Theorem \ref{modular divides stuff}, $4y$ divides $(j_1-j_2)(2y-1)-2y$. So, for some integer $k$, $(j_1-j_2)(2y-1) = 4yk + 2y$. But this implies $2y$ divides $j_1-j_2$, as $\gcd(2y, 2y-1)=1$. Therefore $j_1 = j_2$, so the expression above reads $0\equiv 2xy\pmod{4xy}$, which is a contradiction.

We conclude $i_1 = i_2$, then we reduce the original expression to $(j_1-j_2)(2xy-x)\equiv 0\pmod{n}$. Since $\gcd(4xy, 2xy-x) = x$, by Theorem \ref{modular divides stuff}, $4y$ divides $j_1-j_2$.  So $j_1= j_2$.
}

{
\textit{Columns:} 
Assume two cells along the same column $j$ contain the same entry, so that $m_{i_1, j}\equiv m_{i_2, j}\pmod{n}$. We show $i_1 = i_2$ by setting their expressions equal. Assuming $i_1 < i_2$, and cancelling identical terms we reduce to:
\[m_{i_1, j}\equiv m_{i_2, j}\pmod{n}\iff (i_1-i_2)(2xy) \equiv \sum_{k=i_1 + 1}^{i_2}R(k) \pmod{n}.\]
If $i_2-i_1$ is odd, $\gcd\left((i_1-i_2)(2xy),4xy\right)=2xy$ and if $i_2-i_1$ is even, $\gcd\left((i_1-i_2)(2xy),4xy\right)=4xy$. In either case, by Theorem \ref{modular divides stuff}, $2xy$ must divide the sum on the right. We show this implies $i_2=i_1$ by splitting the sum by cell type:
\[\sum_{k=i_1+1}^{i_2}R(k) = \sum_{\rm Type~ I}(-x) + \sum_{\rm Type~ II}(1) + \sum_{\rm Type~ III}(-1) = I_1 + I_2 + I_3.\]
Where $I_1$ is the sum over Type I cells, $I_2$ is the sum over over Type II, and $I_3$ is the sum over over Type III. 
 
Any column intersects $2y$ blocks, and since row 1 is Type 0, each column contains $2y-1$ Type I cells. Thus, $-2xy+x \leq I_1 \leq 0$. Additionally, there are exactly $x-1$ Type II cells or $x-1$ Type III cells in a block. Furthermore, since the Type II and Type III cells have opposite values for $R(k)$, in any adjacent pair of bands, $B_i,B_{i+1}$:
$$\displaystyle \sum_{{\rm Type~ II, III~in~}B_i,B_{i+1}}R(k)=0.$$ Thus, $-(x-1) \leq I_2+ I_3 \leq x-1$. Therefore,
\[-2xy+1 \leq I_1 + I_2 + I_3 \leq x-1.\]
Since we showed $2xy$ divides the sum, 
\[I_1 +I_2 + I_3 = \sum_{k=i_1+1}^{i_2}R(k) = 0.\]
Then $I_1 = -(I_2 + I_3)$. Since $x$ divides $I_1$ and $-x < I_2+ I_3 < x$, we must have that $I_1 = 0 = I_2 +I_3$. But since $I_1=0$, $i_1$ and $i_2$ intersect the same block. Thus, $i_1=i_2$ since blocks are Latin.
}
\end{proof}

\begin{theorem}\label{even,even}
For even $a,b$, $\MmD(L_{a,b}) \geq \frac{n-a}{2}$.
\end{theorem}
\begin{proof}
Using Algorithm \ref{algo:EvenEven}, the adjacent distances between any two cells along the same row are $\frac{n-a}{2}$ or  $\frac{n-a}{2} + 1$ (every 2 stacks or $4y$ cells). The adjacent distances between any two cells along the same column are $\frac{n-a}{2}$ (between bands), $\frac{n}{2}$ (from an odd to an even row), and $\frac{n}{2} \pm 1$ (everywhere else). Since $a=2x\geq 4$, the minimum of these is $\frac{n-a}{2}$. 
\end{proof}

It may appear that in general $\frac{n-a}{2}$ is the \MaxMinDist~for $(a,b)$-Sudoku Latin squares. However, for the following cases, we have been unable to reach this limit. These cases highlight another limitation of Algorithm \ref{generalized algorithm}.

\begin{theorem}\label{odd,0mod4}
Let $a \leq b$, where $a$ is odd, and $b\equiv 0\pmod{4}$. Then $\MmD (L_{a,b})\geq \dfrac{n-\min\{2a,b\}}{2}$.
\end{theorem}

\begin{proof}
   If $b\equiv 0\pmod{4}$, then $\gcd(b,b-2) = 2$. Additionally, $\gcd(n,\frac{n-2a}{2}) = \gcd\left(ab,a\left(\frac{b-2}{2}\right)\right)=a$. From the proof of Theorem \ref{a by odd}, $\gcd(n, \frac{n-2}{2}) = 1$.  Then using Algorithm \ref{generalized algorithm} with $c = \frac{n-2a}{2}$, $r = \frac{n-2}{2}$, $\beta = 1$, and $\alpha=n$, the resulting square is a Latin square with distance $\frac{n-2a}{2}$. 
   
    If $b < 2a$, then $\frac{n-b}{2} > \frac{n-2a}{2}$.  From Theorem \ref{gcd stuff}, $\gcd(n, \frac{n-b}{2}) = b$.
     Therefore, using Algorithm \ref{generalized algorithm} with $r = \frac{n-b}{2}$, $c = \frac{n-2}{2}$, $\alpha = 1$, and $\beta=n$, the resulting square is a Latin square with distance $\frac{n-b}{2}$.
     
     The proof the Latin squares above are $(a,b)$-Sudoku Latin squares is similar to the proof of Theorem \ref{a by odd}. 
\end{proof}

\begin{theorem}\label{odd,2mod4}
Let $a \leq b$, where $a$ is odd, and $b\equiv 2\pmod{4}$. Then $\MmD (L_{a,b})\geq \dfrac{n-\min\{4a,b\}}{2}$.
\end{theorem}
\begin{proof}
    If $b\equiv 2\pmod{4}$, then $\gcd(b,b-4) = 2$. Additionally, $\gcd(n,\frac{n-4a}{2}) = \gcd\left(ab,a\left(\frac{b-4}{2}\right)\right)=a$. From the proof of Theorem \ref{a by odd}, $\gcd(n, \frac{n-4}{2}) = 1$.  Therefore, using Algorithm \ref{generalized algorithm}, with $c = \frac{n-4a}{2}$, $r = \frac{n-4}{2}$, $\beta = 1$, and $\alpha=n$, the resulting square is a Latin square with distance $c = \frac{n-4a}{2}$. This distance can be improved if $b < 4a$, where we take $c = \frac{n-4}{2}$, $r = \frac{n-b}{2}$, $\alpha = 1$, and $\beta=n$.
   
   The proof the Latin squares above are $(a,b)$-Sudoku Latin squares is similar to the proof of Theorem \ref{a by odd}. 
\end{proof}

We have lower bounds for the \MaxMinDist~of $(a,b)$-Sudoku Latins squares for every possible case of $a,b$. Here, we discuss specific cases (mainly small cases) and attach even more restrictions on them.

\begin{lemma}\label{odd by odd at least 5}
Let $a,b$ be odd and $5\leq a\leq b$. Then $\MmD(L_{a,b}) \leq \frac{n-5}{2}$.
\end{lemma}
\begin{proof}
If $a,b\geq 5$, in any $a\times b$ block we can find a $5\times 5$ subregion, as in Figure \ref{5 by 5 region}.

\begin{figure}[H]
\centering
\begin{tikzpicture}[scale=.6]
\draw(0,0)grid(5,5); 
\draw[step=5,ultra thick](0,0)grid(5,5);
\foreach\x[count=\i] in{ }{\node at(\i-0.5,4.5){$\x$};};
\foreach\x[count=\i] in{ ,  , i, ,  }{\node at(\i-0.5,3.5){$\x$};};
\foreach\x[count=\i] in{ , j, u, \ell, }{\node at(\i-0.5,2.5){$\x$};};
\foreach\x[count=\i] in{ , , k, , }{\node at(\i-0.5,1.5){$\x$};};
\foreach\x[count=\i] in{}{\node at(\i-0.5,0.5){$\x$};};
\end{tikzpicture}
\caption{A $5 \times 5$ region contained in an $a\times b$ block. $i,j,k,\ell,$ and $u$ must all be distinct integers, as well as the cells around them.}
\label{5 by 5 region}
\end{figure}
Without loss of generality, let $u= n$. Assume the \innerDist~of the $(a,b)$-Sudoku Latin square is $\frac{n-3}{2}$, the upper bound from Lemma \ref{theoretical upper bound}. Then $n$ has exactly four neighbors, $\{i,j,k,\ell\}=\{n \pm \frac{n-1}{2} \pmod{n},n \pm \frac{n-3}{2}\pmod{n} \}$ adjacent to it. Without loss of generality, let $i = \frac{n-3}{2}$.

Then $-\frac{n-3}{2}$ is not $j$ or $\ell$. To see this, note that $i$ and $j$ are diagonally adjacent. The cells that are adjacent to both must hold integers that are the proper distance from $i$ and $j$. However, for odd $n$, there is only one integer at this distance from both $\pm \frac{n-3}{2}$, namely $n$. Then $-\frac{n-3}{2}$ must be opposite to $\frac{n-3}{2}$.

The choice of where to place $\pm \frac{n-1}{2}$ is arbitrary, so we place $\frac{n-1}{2}$ on the left, in cell $j$. To summarize, $i = \frac{n-3}{2}$, $j = \frac{n-1}{2}$, $\ell = -\frac{n-1}{2}$, and $k = -\frac{n-3}{2}$.

Using a similar argument, we fill in the four remaining cells in the inner $3\times 3$ region. 

\begin{figure}[H]
\centering
\begin{tikzpicture}[scale=1]
\draw(0,0)grid(5,5); 
\draw[step=5,ultra thick](0,0)grid(5,5);
\foreach\x[count=\i] in{}{\node at(\i-0.5,4.5){$\x$};};
\foreach\x[count=\i] in{, w, \frac{n-3}{2}, x,  }{\node at(\i-0.5,3.5){$\x$};};
\foreach\x[count=\i] in{X, \frac{n-1}{2}, n, -\frac{n-1}{2}, Y}{\node at(\i-0.5,2.5){$\x$};};
\foreach\x[count=\i] in{, y, -\frac{n-3}{2}, z, }{\node at(\i-0.5,1.5){$\x$};};
\foreach\x[count=\i] in{}{\node at(\i-0.5,0.5){$\x$};};
\end{tikzpicture}
\caption{The updated region}
\label{5 by 5 region part 2}
\end{figure}
To do so, we define a function $f: S \to \mathscr{P}(S)$ where $S = \{1,\dots, n\}$ and $\mathscr{P}(S)$ denotes the power set of $S$. Let $f(s) = \{s\pm \frac{n-3}{2}, s\pm \frac{n-1}{2}\}$. This is the set of possible values for the cell adjacent to $s$. To find the values for $w,x,y,$ and $z$, it is helpful to note that $f(\frac{n-3}{2})=\{n-3, n-2, n-1, n\}$, $f(-\frac{n-3}{2})=\{n, 1, 2, 3\}$, $f(\frac{n-1}{2})=\{n-2, n-1, n, 1\}$, and $f(-\frac{n-1}{2})=\{n-1, n, 1, 2\}$. So the value for $x$ in Figure \ref{5 by 5 region part 2} must be an element of $f(\frac{n-3}{2})\cap f(-\frac{n-1}{2}) = \{n-1, n\}$. Since $n$ is in the center, $x= n-1$. 

Similarly, $w = n-2$, $y = 1$, $z = 2$. However, since $f(-\frac{n-1}{2})= \{n-1, n, 1,2\}$, and all four of these integers are already used in this block, $Y$ has no possible value that keeps the inner distance at least $\frac{n-3}{2}$. The same holds for $X$. This argument still holds if we had chosen $k$, $j$, or $\ell$ to be $\frac{n-3}{2}$, with the only difference being the orientation of the 5 by 5 region. 
\end{proof}

Lemma \ref{odd by odd at least 5} is the natural continuation of our argument used to prove that the \MaxMinDist~of Latin squares is $\floor{\frac{n-1}{2}}$, and that the \MaxMinDist~of $(a,b)$-Sudoku Latin squares with $a,b\geq 3$ is at most $\floor{\frac{n-3}{2}}$. This method becomes much more complicated when $n$ is allowed to be even, and much more still when $a,b\geq 7$. However, the following results seem to hint that in general, $\frac{n-a}{2}$ is the \MaxMinDist~for odd $a,b$. 

\begin{theorem}\label{(5,odd)}
If $b\geq 5$ is odd, then $\MmD(L_{5, b}) = \frac{n-5}{2}$.
\end{theorem}
\begin{proof}
For odd $b$, Theorem \ref{a by odd} establishes $\MmD(L_{5, b}) \geq \frac{n-5}{2}$. Lemma \ref{odd by odd at least 5} establishes $\MmD(L_{5, b}) \leq \frac{n-5}{2}$.
\end{proof}

\begin{theorem}\label{(3,odd)}
If $b\geq 3$ is odd or $b=4$, then $\MmD(L_{3, b}) = \left\lfloor\frac{n-3}{2}\right\rfloor$.
\end{theorem}
\begin{proof}
Lemma \ref{theoretical upper bound} establishes $\MmD(L_{3, b})\leq \floor{\frac{n-3}{2}}$ for all cases. For odd $b$, Theorem \ref{a by odd} establishes\\ $\MmD(L_{3, b}) \geq \frac{n-3}{2}=\left\lfloor\frac{n-3}{2}\right\rfloor$. For $b=4$, Theorem \ref{odd,0mod4} establishes $\MmD(L_{3, 4}) \geq \frac{12-\min\{6,4\}}{2}=\floor{\frac{12-3}{2}}$.
\end{proof}

\section{Acknowledgements}
Thank you to Moravian College for hosting the Research Experience for Undergraduates on Research Challenges of Computational and Experimental Mathematics during the summer of 2020. We give special thanks to Dr. Nathan Shank, Dr. Eugene Fiorini, Dr. Tony Wong, and Dr. Joshua Harrington for all the support, advice, and encouragement. Thanks to the National Science Foundation for grant \GrantNumber~for funding this research.

\end{document}